\documentclass[12pt,a4]{article}
\usepackage[centertags]{amsmath}
\usepackage{amsfonts}
\usepackage{amssymb}
\usepackage{amsthm}
\usepackage{newlfont}

\newcommand{\Z}{\Bbb Z}

\newcommand{\R}{\Bbb R}
\def\NL{\hfill\break}

\newcommand{\set}[1]{\left\{#1\right\}}
\newcommand{\parth}[1]{\left(#1\right)}

\newtheorem{theo}{Theorem}[section]
\newtheorem{cor}[theo]{Corollary}
\newtheorem{lem}[theo]{Lemma}
\newtheorem{prop}[theo]{Proposition}
\newtheorem{ex}[theo]{Example}
\newtheorem{df}[theo]{Definition}

\begin{document}
\centerline {\Large{\bf On McCoy Condition and Semicommutative Rings}}

\centerline{}
\centerline{}
\centerline{\bf Mohamed Louzari}
\centerline{}
\centerline{Department of mathematics, Faculty of sciences,}
\centerline{Abdelmalek Essaadi University, Tetouan B.P 2121, Morocco}
\centerline{e-mail: mlouzari@yahoo.com}
\centerline{}
\date{\today}

\begin{abstract}Let $R$ be a ring, $\sigma$ an endomorphism of $R$, $I$ a right ideal in $S=R[x;\sigma]$ and $M_R$ a right $R$-module. We give a generalization of McCoy's Theorem \cite{mccoy}, by showing that, if $r_S(I)$ is $\sigma$-stable or $\sigma$-compatible. Then $\;r_S(I)\neq 0$ implies $r_R(I)\neq 0$. As a consequence, if $R[x;\sigma]$ is semicommutative then $R$ is $\sigma$-skew McCoy. Moreover, we show that the Nagata extension $R\oplus_{\sigma}M_R$ is semicommutative right McCoy when $R$ is a commutative domain.
\end{abstract}

{\bf Mathematics Subject Classification:} 16S36, 16U80\\

{\bf Keywords:} Armendariz rings, McCoy rings, Nagata extension, semicommutative rings, $\sigma$-skew McCoy.

\section{Introduction}
Throughout the paper, $R$ will always denote an associative ring with identity and $M_R$ will stand for a right R-module. Given a ring $R$, the polynomial ring with an indeterminate $x$ over $R$ is denoted by $R[x]$. According to Nielson \cite{nielson/2006}, a ring $R$ is called {\it right McCoy} (resp., {\it left McCoy}) if, for any polynomials $f(x),g(x)\in R[x]\setminus \{0\}$, $f(x)g(x)=0$ implies $f(x)r=0$ (resp., $sg(x)=0$) for some $0\neq r\in R$ (resp., $0\neq s\in R$). A ring is called {\it McCoy} if it is both left and right McCoy. By McCoy \cite{mccoy2}, commutative rings are McCoy rings. Recall that a ring $R$ is {\it reversible} if $ab=0$ implies $ba=0$ for $a,b\in R$, and $R$ is {\it semicommutative} if $ab=0$ implies $aRb=0$ for $a,b\in R$. It is obvious that commutative rings are reversible and reversible rings are semicommutative, but the converse do not hold, respectively. Recently, Nielson \cite{nielson/2006}, proved that reversible rings are McCoy. In \cite[Corollary 2.3]{hirano/2002}, it was claimed that all semicommutative rings were McCoy. However, Hirano's claimed assumed that if $R$ is semicommutative then $R[x]$ is semicommutative, but this was later shown to be false \cite[Example 2]{huh/2002}. Nielson \cite[Theorem 2]{nielson/2006} shows that reversible rings are McCoy and he gives an example of a semicommutative ring which is not right McCoy. Recall that a ring is {\it reduced} if it has no nonzero nilpotent elements. A module $M_R$ is called {\it Armendariz} if whenever polynomials $m=\sum_{i=0}^{n}m_ix^i\in M[x]$ and $f=\sum_{j=0}^{m}a_jx^j\in R[x]$ satisfy $mf=0$, then $m_ia_j=0$ for each $i,j$. A ring $R$ is Armendariz, if $R_R$ is Armendariz. Reduced rings are Armendariz and Armendariz rings are McCoy. We have the following diagram:

$$\begin{array}{c}
R\;\mathrm{is\;reversible} \\
R[x]\;\mathrm{is\;semicommutative}\\
R\;\mathrm{is\;Armendariz}
  \end{array}
\left\}\begin{array}{c}
           \\
         \Rightarrow R\;\mathrm{is\;McCoy}\\
         \\
        \end{array}\right.$$

An Ore extension of a ring $R$ is denoted by $R[x;\sigma,\delta]$, where $\sigma$ is an endomorphism of $R$ and $\delta$ is a $\sigma$-derivation, i.e., $\delta\colon R\rightarrow R$ is an additive map such that $\delta(ab)=\sigma(a)\delta(b)+\delta(a)b$ for all $a,b\in R$. Recall that elements of $R[x;\sigma,\delta]$ are polynomials in $x$ with coefficients written on the left. Multiplication in $R[x;\sigma,\delta]$ is given by the multiplication in $R$ and the condition $xa=\sigma(a)x+\delta(a)$, for all $a\in R$. For $\delta=0$, we put $R[x;\sigma,0]=R[x;\sigma]$.  Baser et al. \cite{kwak/2009}, introduced $\sigma$-skew McCoy for an endomorphism $\sigma$ of $R$. A ring $R$ is called {\it $\sigma$-skew McCoy}, if for any nonzero polynomials $p(x)=\sum_{i=0}^na_ix^i$ and $q(x)=\sum_{j=0}^mb_jx^j\in R[x;\sigma]$, $p(x)q(x)=0$ implies $p(x)c=0$ for some nonzero $c\in R$, and they have proved the following:

$$\begin{array}{c}
   R[x;\sigma]\;\mathrm{is\;right\;McCoy} \\
   R[x;\sigma]\;\mathrm{is\;reversible}
  \end{array}
\left\}\begin{array}{c}
           \\
         \Rightarrow R\;\mathrm{is}\;\sigma\mathrm{-skew\;McCoy} \\
         \\
        \end{array}\right.$$

Hong et al. \cite{hong/2010}, proved that if $\sigma$ is an automorphism of $R$ and $I$ a right ideal of $S=R[x;\sigma,\delta]$ then $r_S(I)\neq 0$ implies $r_R(I)\neq 0$, which is a generalization of McCoy's Theorem \cite{mccoy}.
\par In this paper, we give another generalization of McCoy's Theorem, by showing that for any right ideal $I$ of $S=R[x;\sigma]$, we have $r_S(I)\neq 0$ implies $r_R(I)\neq 0$ when $R$ is $\sigma$-compatible or $r_S(I)$ is $\sigma$-stable. As a consequence, if $R[x;\sigma]$ is semicommutative then $R$ is $\sigma$-skew McCoy. We obtain a generalization of \cite[Corollary 6]{kwak/2009} and \cite[Corollary 2.3]{hirano/2002}.
Furthermore, we show some results on Nagata extension. For a commutative ring $R$, we have
\NL$1)$ If $R$ is a domain, then
\par$(a)$ $M_R$ is Armendariz if and only if $R\oplus_{\sigma}M_R$ is Armendariz.
\par$(b)$ The ring $R\oplus_{\sigma}M_R$ is semicommutative and right McCoy.
\NL$2)$ If $R$ and $M_R$ are Armendariz such that $M_R$ satisfies the condition $(\mathcal{C}_{\sigma}^2)$, then $R\oplus_{\sigma}M_R$ is Armendariz.

\section{A Generalization of McCoy's Theorem}

McCoy \cite{mccoy}, proved that for any right ideal $I$ of $S=R[x_1,x_2,\cdots,x_n]$ over a ring $R$, if $\;r_S(I)\neq 0$ then $r_R(I)\neq 0$. This result was extended by Hong et al. \cite{hong/2010} to many skew polynomial rings, where $\sigma$ is an automorphism of $R$. Herein, we'll extend McCoy's Theorem to skew polynomial rings of the form $R[x;\sigma]$ with $\sigma$ an endomorphism of $R$. According to Annin \cite{annin}, a ring $R$ is $\sigma$-{\it compatible}, if for any $a,b\in R$, $ab=0$ if and only if $a\sigma(b)=0$. Let $\sigma$ be an endomorphism of $R$ and $I$ an ideal of $R$, we say that the ideal $I$ is {\it $\sigma$-stable}, if $\sigma(I)\subseteq I$. Let $\sigma$ be an endomorphism of a ring $R$, then for any $f=\sum_{i=0}^na_ix^i\in R[x;\sigma]$, we denote the polynomial $\sigma(f)$ by $\sigma(f)=\sum_{i=0}^n\sigma(a_i)x^i$.

\begin{theo}\label{gen/mccoy}Let $R$ be a ring, $\sigma$ an endomorphism of $R$ and $I$ a right ideal in $S=R[x;\sigma]$. Suppose that $r_S(I)$ is $\sigma$-stable or $\sigma$-compatible. If $\;r_S(I)\neq 0$ then $r_R(I)\neq 0$.
\end{theo}

\begin{proof}Suppose that $r_R(I)\neq 0$. If $I=0$, then it's trivial. Assume that $I\neq 0$. Let $g(x)=\sum_{j=0}^mb_jx^j\in r_S(I)$, we can set $b_m\neq 0$. If $m=0$, we're done, so we can suppose that $m\geq 1$. In this situation, if $Ib_m=0$, then we're done. Otherwise, there exists $0\neq f(x)=\sum_{i=0}^na_ix^i\in I$, such that $f(x)b_m\neq 0$.

\par{\bf Case 1.} If $r_S(I)$ is $\sigma$-stable, then $a_ix^ib_m\neq 0$ for some $i\in\{0,1,\cdots,n\}$, therefore $a_ix^ig(x)\neq 0$. Take $p=\max\{i|a_ix^ig(x)\neq 0\}$. Then $a_p\sigma^p(g(x))\neq 0$ and $a_ix^ig(x)=0$ for $i\geq p+1$. We obtain $a_p\sigma^p(b_m)=0$ from $f(x)g(x)=0$. Also, we have $I(a_p\sigma^p(g(x)))=(Ia_p)\sigma^p(g(x))=0$ because $I$ is a right ideal of $S$ and $\sigma^p(g(x))\in r_S(I)$. So $0\neq a_p\sigma^p(g(x))\in r_S(I)$. We can write $a_p\sigma^p(g(x))=a_p\sigma^p(b_0)+a_p\sigma^p(b_1)x+\cdots+a_p\sigma^p(b_{\ell})x^{\ell}$, where $a_p\sigma^p(b_{\ell})\neq 0$ and $\ell< m$. If $\ell=0$ then $Ia_p\sigma^p(b_{\ell})=0$, so $0\neq a_p\sigma^p(b_{\ell})\in r_R(I)$. Otherwise, $\ell\geq 1$, then we'll consider $a_p\sigma^p(g(x))$ in place of $g(x)$. $0\neq h(x)=\sum_{k=0}^sc_kx^k\in I$ such that $h(x)a_p\sigma^p(b_{\ell})\neq 0$. We can find $q$ as a the largest integer such that $c_q\sigma^q(a_p\sigma^p(g(x)))\neq 0$ and then $0\neq c_q\sigma^q(a_p\sigma^p(g(x)))\in r_S(I)$ such that the degree of $c_q\sigma^q(a_p\sigma^p(g(x)))$ is smaller than one of $a_p\sigma^p(g(x))$.

\par{\bf Case 2.} If $R$ is $\sigma$-compatible, then $a_i\sigma^i(b_m)\neq 0$ for some $i\in\{0,1,\cdots,n\}$, also $a_ib_m\neq 0$, by $\sigma$-compatibility of $R$. Therefore $a_ig(x)\neq 0$ for some $i\in\{0,1,\cdots,n\}$. Take $p=\max\{i|a_ig(x)\neq 0\}$, so $a_pg(x)\neq 0$ and $a_{p+1}g(x)=\cdots=a_ng(x)=0$, thus $a_ib_j=0$ for $i\in\{p+1,\cdots,n\}$ and $j\in\{0,1,\cdots,m\}$. For $i\in\{p+1,\cdots,n\}$, we have $a_ix^ig(x)=\parth{\sum_{j=0}^ma_i\sigma^i(b_j)x^j}x^i=0$, with $\sigma$-compatibility and the previous condition. On other,  we get $a_pb_m=0$ from $f(x)g(x)=0$. So that the degree of $a_pg(x)$ is less than $m$ such that $a_pg(x)\neq 0$. But $I(a_pg(x))=(Ia_p)g(x)=0$ since $I$ is a right ideal of $S$, so $0\neq a_pg(x)\in r_S(I)$. We can write $a_pg(x)=\sum_{k=0}^{\ell}a_pb_kx^k$ with $a_pb_{\ell}\neq 0$ and $\ell<m$. We have the two possibilities: If $\ell=0$ then $a_pg(x)$ is a nonzero element in $r_R(I)$. Otherwise, $\ell\geq 1$. Then we'll consider $a_pg(x)$ in place of $g(x)$. We have two cases $I(a_pb_{\ell})=0$ or $I(a_pb_{\ell})\neq 0$. The first implies $0\neq a_pb_{\ell}\in r_R(I)$, for the second, there exists $0\neq h(x)=\sum_{k=0}^sc_kx^k\in I$ such that $h(x)a_pb_{\ell}\neq 0$. Here, we can find $q$ as a the largest integer such that $c_qa_pg(x)\neq 0$ and then $0\neq c_qa_pg(x)\in r_S(I)$ such that the degree of $c_qa_pg(x)$ is smaller than one of $a_pg(x)$.
\par Continuing with the same manner (in the two cases), we can produce elements of the forms
$0\neq a_{t_1}a_{t_2}\cdots a_{t_s}\sigma^{t_1+t_2+\cdots+t_s}g(x)$ (resp., $0\neq a_{t_1}a_{t_2}\cdots a_{t_s}g(x)$) in $r_S(I)$, with $s\leq m$ and the degree of these polynomials is zero. Thus $a_{t_1}a_{t_2}\cdots a_{t_s}\sigma^{t_1+t_2+\cdots+t_s}g(x)\in r_R(I)$ (resp., $0\neq a_{t_1}a_{t_2}\cdots a_{t_s}g(x)\in r_R(I))$. Therefore $r_R(I)\neq 0$.
\end{proof}

\begin{cor}[{\cite[Theorem 2.2]{hirano/2002}}]Let $f(x)$ be an element of $R[x]$. If $\;r_{R[x]}(f(x)R[x])\neq 0$ then $r_{R[x]}(f(x)R[x])\cap R\neq 0$
\end{cor}

\begin{proof}Consider the right ideal $I=f(x)R[x]$.
\end{proof}

\begin{cor}Let $R$ be a ring, $\sigma$ an endomorphism of $R$ and $I$ a right ideal of $S=R[x;\sigma]$. If $S$ is semicommutative, then $r_S(I)\neq 0$ implies $r_R(I)\neq 0$.
\end{cor}

According to Clark \cite{clark}, a ring $R$ is said to be {\it quasi-Baer} if the right annihilator of each right ideal of  $R$ is generated (as a right ideal) by an idempotent. Following Zhang and Chen \cite{zhang/chen}, a ring $R$ is said to be $\sigma$-semicommutative if, for any $a,b\in R$, $ab=0$ implies $aR\sigma(b)=0$. A ring $R$ is called {\it right $($left$)$ $\sigma$-reversible} if whenever $ab=$ for $a,b\in R$, $b\sigma(a)=0$ ($\sigma(b)a=0$). A ring $R$ is called $\sigma$-reversible if it is both right and left $\sigma$-reversible. Hong et al. \cite{hong/2000}, proved that, if $R$ is $\sigma$-rigid then $R$ is quasi-Baer if and only if $R[x;\sigma]$ is quasi-Baer. Recently, Hong et al. \cite{hong/2009}, have proved the same result when $R$ is semi-prime and all ideals of $R$ are $\sigma$-stable.

\begin{prop}\label{cor1}Let $R$ be a $\sigma$-semicommutative ring. If $R[x;\sigma]$ is quasi-Baer then $R$ so is.
\end{prop}

\begin{proof} Let $I$ be a right ideal of $R$. We have $r_{R[x;\sigma]}(IR[x;\sigma])=eR[x;\sigma]$ for some idempotent $e=e_0+e_1x+\cdots+e_mx^m\in R[x;\sigma]$. By \cite[Proposition 3.9]{baser/2008}, $r_R(IR[x;\sigma])=e_0R$. Clearly, $r_R(IR[x;\sigma])\subseteq r_R(I)$. Conversely, let $b\in r_R(I)$ then $Ib=0$, since $R$ is $\sigma$-semicommutative, we have $IR[x;\sigma]b=0$, so $b\in r_R(IR[x;\sigma])$. Therefore $r_R(I)=e_0R$.
\end{proof}

\begin{ex}Let $\Z$ be the ring of integers and consider the ring $$R=\set{(a,b)\in\Z\oplus\Z\mid\;a\equiv b\;\mathrm{(mod\;2)}}$$ and $\sigma\colon\R\rightarrow R$ defined by $\sigma(a,b)=(b,a)$.
\NL$1)$ $R[x;\sigma]$ is quasi-Baer and $R$ is not quasi-Baer, by \cite[Example 9]{hong/2000}.
\NL$2)$ $R$ is not $\sigma$-semicommutative. Let $a=(2,0)$ and $b=(0,2)$ we have $ab=0$, but $a\sigma(b)=(2,0)(2,0=(4,0)\neq 0)$. Thus $R$ is not semicommutative. Therefore the condition ``$R$ is $\sigma$ semicommutative" is not a superfluous condition in Proposition \ref{cor1}.
\end{ex}

\begin{prop}\label{semicomm/Mccoy}Let $\sigma$ be an endomorphism of a ring $R$. If $R[x;\sigma]$ is a semicommutative ring then $R$ is $\sigma$-skew McCoy.
\end{prop}

\begin{proof}Let $f(x)\in R[x;\sigma]$ and consider the right ideal $I=f(x)R[x;\sigma]$. We claim that $r_{R[x;\sigma]}(f(x)R[x;\sigma])$ is $\sigma$-stable. Let $\varphi(x)\in R[x;\sigma]$ and $h(x)\in r_{R[x;\sigma]}(f(x)R[x;\sigma])$. Then $f(x)\varphi(x)h(x)=0$, since $R[x;\sigma]$ is semicommutative, we get $f(x)\varphi(x)xh(x)=0$ which gives $f(x)\varphi(x)\sigma(h(x))x=0$. Therefore $r_{R[x;\sigma]}(f(x)R[x;\sigma])$ is $\sigma$-stable. On other hand, let $g(x)\in R[x;\sigma]\setminus\{0\}$ such that $f(x)g(x)=0$. We have $f(x)R[x;\sigma]g(x)=0$, then there exists a nonzero $c\in R$ satisfying $f(x)R[x;\sigma]c=0$ by Theorem \ref{gen/mccoy}. In particular, we have $f(x)c=0$. Thus $R$ is $\sigma$-skew McCoy.

\end{proof}

We have immediately the next Corollaries.

\begin{cor}[{\cite[Corollary 6]{kwak/2009}}]Let $\sigma$ be an endomorphism of a ring $R$. If $R[x;\sigma]$ is reversible then $R$ is $\sigma$-skew McCoy.
\end{cor}

\begin{cor}[{\cite[Corollary 2.3]{hirano/2002}}]Let $\sigma$ be an endomorphism of a ring $R$. If $R[x]$ is semicommutative then $R$ is right McCoy.
\end{cor}

\begin{df}\label{df1}Let $R$ be a ring, $M_R$ an $R$-module and $\sigma$ an endomorphism of $R$. For $m\in M_R$ and $a\in R$. We say that $M_R$ satisfies the condition  $(\mathcal{C}_{\sigma}^1)$ $($resp., $(\mathcal{C}_{\sigma}^2)$$)$ if $ma=0$ $($resp., $m\sigma(a)a=0$$)$ implies $m\sigma(a)=0$.
\end{df}

\begin{cor}\label{cor2}Let $\sigma$ be an endomorphism of a ring $R$.
\NL$(1)$ If $R$ is semicommutative satisfying the condition $(\mathcal{C}_{\sigma}^2)$ then it is $\sigma$-skew McCoy.
\NL$(2)$ If $R$ is reduced and right $\sigma$-reversible then it is $\sigma$-skew McCoy.
\end{cor}

\begin{proof}$(1)$ Immediately from \cite[Proposition 3.4]{louzari/2011}. $(2)$ Clearly from $(1)$.
\end{proof}

There is an example of a ring $R$ and an endomorphism $\sigma$ of $R$, such that $R[x;\sigma]$ is semicommutative and $R$ is $\sigma$-skew McCoy.

\begin{ex}\label{counter/ex1}Consider the ring $R=\Z_2[x]$ where $\Z_2$ is the ring of integers modulo $2$ and $\sigma$ an endomorphism of $R$ defined by $\sigma(f(x))=f(0)$.
\NL$(1)$ By \cite[Example 5]{hong/2003}, $R$ is $\sigma$-skew Armendariz then it's $\sigma$-skew McCoy.
\NL$(2)$ $R$ is a commutative domain so it's Baer then right p.q.-Baer (\cite[Example 8]{hong/2006}).
\NL$(3)$ Since $R$ is semicommutative and right p.q.-Baer, if we show that $R$ satisfies $(\mathcal{C}_{\sigma}^1)$ then by \cite[Corollary 3.5]{louzari/2011}, $R[y,\sigma]$ is semicommutative. Let $f=a_0+a_1x+\cdots+a_nx^n$ and $g=b_0+b_1x+\cdots+b_mx^m\in R$. Suppose that $fg=0$ then we have the system of equations:
$$a_0b_0=0\qquad\qquad\qquad\quad\qquad\eqno(0)$$
$$a_0b_1+a_1b_0=0\qquad\qquad\qquad\eqno(1)$$
$$a_0b_2+a_1b_1+a_2b_0=0\qquad\quad\eqno(2)$$
$$a_0b_3+a_1b_2+a_2b_1+a_3b_0=0\eqno(3)$$
$$\vdots\qquad\qquad\qquad\qquad\;\;$$
$$a_nb_m=0\qquad\qquad\qquad\qquad\eqno(n+m)$$
Eq.$(0)$ implies $a_0=0$ or $b_0=0$, if $b_0=0$ then $f\sigma(g)=fb_0=0$. We can suppose $b_0\neq 0$. Then $a_0=0$. Eq.$(1)$ implies $a_1b_0=0$. Eq.$(2)$ implies $a_1b_1+a_2b_0=0$, multiplying on the right by $b_0$, we find $a_1b_1b_0+a_2b_0^2=0$ then $a_2b_0^2=0$ because $a_1b_1b_0=0$ $($since $\Z_2$ is semicommutative$)$, also $\Z_2$ is reduced then $a_2b_0=0$. We continue with the same manner yields $a_ib_0=0$ for $i=0,1,\cdots,n$. Thus $fb_0=f\sigma(g)=0$. Therefore $R$ satisfies the condition $(\mathcal{C}_{\sigma}^1)$.

\end{ex}

\begin{ex}Consider the ring $R=\Z_2\oplus\Z_2$ with the usual addition and multiplication. Let $\sigma\colon R\rightarrow R$ be defined by $\sigma(a,b)=(b,a)$. Consider $p(x)=(1,0)+(1,0)x$ and $q(x)=(0,1)+(1,0)x\in R[x;\sigma]$, we have $p(x)q(x)=0$. But $p(x)c\neq 0$ for any nonzero $c\in R$. Thus $R$ is not $\sigma$-skew McCoy. Also $p(x)(1,0)q(x)=(1,0)x\neq 0$.Therefore $R[x;\sigma]$ is not semicommutative.
\end{ex}

\begin{ex}Let $\Z_3$ be the ring of integers modulo $3$. Consider the $2\times 2$ matrix ring $R=Mat_2(\Z_3)$ over $\Z_3$ and an endomorphism $\sigma\colon R\rightarrow R$ be defined by
 $$\sigma\parth{\left(
\begin{array}{cc}
a & b \\
c & d \\
\end{array}
\right)
}=\left(
\begin{array}{cc}
a & -b \\
-c & d \\
\end{array}
\right).
$$
\NL$(1)$ By \cite[Example 11]{kwak/2009}, $R$ is not $\sigma$-skew McCoy.
\NL$(2)$ $R[x;\sigma]$ is not semicommutative. For
$$p(x)=\left(
\begin{array}{cc}
1 & 0 \\
0 & 0 \\
\end{array}
\right)+\left(
\begin{array}{cc}
1 & 1 \\
0 & 0 \\
\end{array}
\right)x\;,\quad q(x)=\left(
\begin{array}{cc}
0 & 0 \\
0 & -1 \\
\end{array}
\right)+\left(
\begin{array}{cc}
0 & 1 \\
0 & 1 \\
\end{array}
\right)x\in R[x;\sigma].$$
We have $p(x)q(x)=0$, but $p(x)\left(
\begin{array}{cc}
1 & 1 \\
0 & 0 \\
\end{array}
\right)q(x)\neq 0$.
\end{ex}

\section{Nagata Extension and McCoyness}

The next construction is due to Nagata \cite{nagata}. Let $R$ be a commutative ring, $M_R$ be  an $R$-module and $\sigma$ an endomorphism of $R$. The $R$-module $R\oplus_{\sigma}M_R$ acquires a ring structure (possibly noncommutative), where the product is defined by $(a,m)(b,n)=(ab,n\sigma(a)+mb)$, where $a,b\in R$ and $m,n\in M_R$. We shall call this extension the {\it Nagata extension} of $R$ by $M_R$ and $\sigma$. If $\sigma=id_R$, then $R\oplus_{id_R} M_R$ (denoted by $R\oplus M_R$) is a commutative ring. Anderson and Camillo \cite{anderson/camillo/1998}, have proved that if $R$ is commutative domain then $M_R$ is Armendariz if and only if $R\oplus M_R$ is Armendariz. We'll see that this result still true for $R\oplus_{\sigma} M_R$. Kim et al. \cite{kim/2003}, have proved that, if $R$ is a commutative domain and $\sigma$ is a monomorphism of $R$ then $R\oplus_{\sigma}R$ is reversible, and so it is McCoy. Recall that if $\sigma$ is an endomorphism of a ring $R$, then the map $R[x]\rightarrow R[x]$ defined by $\sum_{i=0}^{n}a_ix^i\mapsto\sum_{i=0}^{n}\sigma(a_i)x^i$ is an endomorphism of the polynomial ring $R[x]$ and clearly this map extends $\sigma$. We shall also denote the extended map $R[x]\rightarrow R[x]$ by $\sigma$ and the image of $f\in R[x]$ by $\sigma(f)$. In this section, we'll discuss when the Nagata extension $R\oplus_{\sigma}M_R$ is McCoy. Let $R$ be a commutative domain. The set $T(M)=\{m\in M|r_R(m)\neq 0\}$ is called the {\it torsion submodule} of $M_R$. If $T(M)=M$ (resp., $T(M)=0$) then $M_R$ is {\it torsion} (resp., {\it torsion-free}).

\begin{prop}\label{prop2}Let $R$ be a commutative domain and $M_R$ an $R$-module. Then $R\oplus_{\sigma} M_R$ is Armendariz if and only if $M_R$ is Armendariz. In particular, if $M_R$ is torsion-free then $R\oplus_{\sigma} M_R$ is Armendariz.
\end{prop}

\begin{proof}Let $R'=R\oplus_{\sigma} M_R$, we have $R'[x]=R[x]\oplus_{\sigma}M[x]$. Suppose that $R'$ is Armendariz. Let $m=\sum_{i=0}^{p}m_ix^i\in M[x]$ and $f=\sum_{j=0}^{q}a_jx^j\in R[x]$ with $mf=0$. We have $(0,m)=\sum_{i=0}^{p}(0,m_i)x^i\in R'[x]$ and $(f,0)=\sum_{j=0}^{q}(a_j,0)x^j\in R'[x]$, since $R'$ is Armendariz then $(0,m_i)(a_j,0)=(0,m_ia_j)=(0,0)$ for all $i,j$. Thus $m_ia_j=0$ for all $i,j$. Conversely, suppose that $M_R$ is Armendariz. Let $f,g\in R[x]$ and $m,n\in M[x]$ such that $(f,m)(g,n)=(0,0)$. Write $(f,m)=\sum(a_i,m_i)x^i\in R'[x]$ and $(g,n)=\sum(b_j,n_j)x^j\in R'[x]$. From $(f,m)(g,n)=(0,0)$, we have $(fg,n\sigma(f)+mg)=(0,0)$. Since $R[x]$ is a commutative domain, then $f=0$ or $g=0$. If $f=0$, we get $mg=0$. Then $m_ib_j=0$ and $a_i=0$ for all $i,j$. Thus $(a_i,m_i)(b_j,n_j)=(a_ib_j,n_j\sigma(a_i)+m_ib_j)=(0,0)$. Otherwise, we get $n\sigma(f)=0$. Then $b_j=0$ and $n_j\sigma(a_i)=0$ for all $i,j$. Thus $(a_i,m_i)(b_j,n_j)=(a_ib_j,n_j\sigma(a_i)+m_ib_j)=(0,0)$. Therefore $R'$ is Armendariz.
\end{proof}

\begin{cor}Let $R$ be a commutative domain and $M_R$ an $R$-module satisfying the condition $(\mathcal{C}_{id_R}^2)$. Then $R\oplus_{\sigma} M_R$ is Armendariz.
\end{cor}

\begin{proof}Since $M_R$ is semicommutative then it is Armendariz by \cite[Lemma 3.3]{louzari/2011}.
\end{proof}

\begin{prop}\label{prop1}Let $R$ be a commutative ring and $M_R$ an $R$-module such that $R$ satisfies $(\mathcal{C}_{\sigma}^1)$ and $M_R$ satisfies $(\mathcal{C}_{\sigma}^2)$. Then $R\oplus_{\sigma}M_R$ is a semicommutative ring.
\end{prop}

\begin{proof}We'll use freely the conditions $(\mathcal{C}_{\sigma}^1)$ and $(\mathcal{C}_{\sigma}^2)$. Let $(r,m),(s,n)\in R\oplus_{\sigma}M_R$ such that $$(r,m)(s,n)=(rs,n\sigma(r)+ms)=(0,0).\eqno(1)$$ We'll show that for any $(t,u)\in R\oplus_{\sigma}M_R$ $$(r,m)(t,u)(s,n)=(rts,n\sigma(rt)+u\sigma(r)s+mts)=(0,0).\eqno(2)$$ It suffices to show $n\sigma(rt)+u\sigma(r)s+mts=0$. Multiplying $n\sigma(r)+ms=0$ of Eq.$(1)$ on the right hand by $r$, gives $n\sigma(r)r=0$, so we get $n\sigma(r)=0$ and hence $ms=0$. Thus $n\sigma(rt)=mts=0$. Clearly $rs=0$ implies $\sigma(r)s=0$ and so $u\sigma(r)s=0$. Therefore $n\sigma(rt)+u\sigma(r)s+mts=0$.
\end{proof}

\begin{prop}\label{prop3}Let $R$ be commutative domain and $M_R$ an $R$-module. Then $R\oplus_{\sigma}M_R$ is a semicommutative right McCoy ring.
\end{prop}

\begin{proof}Consider equations $(1)$ and $(2)$ of Proposition \ref{prop1}. From Eq.$(1)$, we get $r=0$ or $s=0$ since $R$ is a domain. Say $r=0$, then $rts=n\sigma(rt)=u\sigma(r)s=0$, and $mts=0$ from $(1)$, hence we have $(2)$. Next say $s=0$, it follows $rts=u\sigma(r)s=mts=0$ and $n\sigma(rt)=0$ from $(1)$, and so we have $(2)$. Therefore $(r,m)(R\oplus_{\sigma}M)(s,n)=0$. For McCoyness, let $(r,m),(s,n)\in R'=R\oplus_{\sigma}M_R$. Suppose that $(r,m)(s,n)^2=(rs^2,n\sigma(r^2)+ns\sigma(r)+ms^2)=0$, then $r=0$ or $s=0$ which implies $(r,m)(s,n)=(rs,n\sigma(r)+ms)=0$. Thus by Corollary \ref{cor2}(1), $R\oplus_{\sigma}M_R$ is right McCoy.
\end{proof}

The next example shows that under the conditions of Proposition \ref{prop3}, $R\oplus_{\sigma}M_R$ can't be reversible.

\begin{ex}Let $D$ be a commutative domain and $R=D[x]$ be the polynomial ring over $D$ with an indeterminate $x$. Consider the endomorphism $\sigma\colon R\rightarrow R$ defined by $\sigma(f(x))=f(0)$. Since $(x,1)(0,1)=(0,0)$ and $(0,1)(x,1)=(0,x)\neq (0,0)$, then $R\oplus_{\sigma}R$ is not reversible. Thus $R\oplus_{\sigma}M_R$ can't be reversible under the conditions of Proposition \ref{prop3}.
\end{ex}

\begin{lem}\label{lem2}If $M_R$ is an Armendariz module. Let $m(x)\in M[x]$ and $f(x), g(x)\in R[x]$ such that $m(x)=\sum_{i=0}^nm_ix^i$, $f(x)=\sum_{j=0}^pa_jx^j$ and $g(x)=\sum_{k=0}^qb_kx^k$. Then $$m(x)f(x)g(x)=0\Leftrightarrow m_ia_jb_k=0\;\mathrm{for\;all}\;i,j,k.$$
\end{lem}

\begin{proof}($\Leftarrow$) Clear. ($\Rightarrow$) If $m(x)f(x)=0$ then $m(x)a_j=0$ for all $j$. Now, if $m(x)f(x)g(x)=0$ then $m(x)[f(x)b_k]=0$ for all $k$. Since $M_R$ is Armendariz we have $m_i(a_jb_k)=0$ for all $i,j$. Thus $m_ia_jb_k=0$ for all $i,j,k$.
\end{proof}

\begin{lem}\label{lem3}If $M_R$ is an Armendariz module satisfying the condition $(\mathcal{C}_{\sigma}^2)$. Then $M[x]_{R[x]}$ satisfies the condition $(\mathcal{C}_{\sigma}^2)$.
\end{lem}

\begin{proof}Let $m(x)=\sum_{i=0}^nm_ix^i\in M[x]$ and $f(x)=\sum_{j=0}^pa_jx^j\in R[x]$. Suppose that $m(x)\sigma(f(x))f(x)=0$. By Lemma \ref{lem2}, $m_i\sigma(a_j)a_k=0$ for all $i,j,k$. In particular $m_i\sigma(a_j)a_j=0$ for all $i,j$. Then $m_i\sigma(a_j)=0$ for all $i,j$. Therefore $m(x)\sigma(f(x))=0$.
\end{proof}

\begin{theo}\label{theo2}Let $R$ be a commutative Armendariz ring, $\sigma$ an endomorphism of $R$ and $M_R$ a module satisfying the condition $(\mathcal{C}_{\sigma}^2)$. Then $M_R$ is Armendariz if and only if $R\oplus_{\sigma}M_R$ is Armendariz.
\end{theo}

\begin{proof}Let $f,g\in R[x]$ and $m,n\in M[x]$ such that $(f,m)(g,n)=(0,0)$. Write $(f,m)=\sum(a_i,m_i)x^i\in R'[x]$ and $(g,n)=\sum(b_j,n_j)x^j\in R'[x]$. From $(f,m)(g,n)=(0,0)$, we have $(fg,n\sigma(f)+mg)=(0,0)$. Since $R$ is Armendariz, then $a_ib_j=0$ for all $i,j$. Multiplying $n\sigma(f)+mg=0$ on the right by $f$, we have $n\sigma(f)f=0$ by Lemma \ref{lem3}, we get $n\sigma(f)=0$ and so $mg=0$. Since $M_R$ is Armendariz we have $m_ib_j=0$ and $n_i\sigma(a_j)=0$ for all $i,j$. Thus $(a_i,m_i)(b_j,n_j)=(a_ib_j,n_j\sigma(a_i)+m_ib_j)=(0,0)$. Therefore $R'$ is Armendariz. The converse is clear.
\end{proof}

\begin{cor}\label{cor1}If $R$ is a commutative reduced ring which satisfies the condition $(\mathcal{C}_{\sigma}^1)$ then $R\oplus_{\sigma}R$ is semicommutative and Armendariz.
\end{cor}

\begin{proof}Immediately by Proposition \ref{prop1} and Theorem \ref{theo2}.
\end{proof}

\begin{ex}Consider the ring $R=\Z_2\oplus\Z_2$ with the usual addition and multiplication. Let $\sigma\colon R\rightarrow R$ be defined by $\sigma(a,b)=(b,a)$. Clearly $R$ is a commutative reduced ring but not a domain. Let $A=((0,1),(0,1))$, $B=((1,0),(0,1))$ and $C=((1,0),(1,0))$. We have $$AB=((0,1),(0,1))((1,0),(0,1))=((0,0),((0,1)\sigma(0,1)+(0,1)(1,0)))=0.$$ But
$$ACB=((0,1),(0,1))((1,0),(1,0))((1,0),(0,1))=((0,0),(1,0))((1,0),(0,1))$$
$$\qquad\qquad\qquad\qquad\qquad\qquad\qquad\qquad\quad=((0,0),(1,0))\neq 0.$$ Hence $R\oplus_{\sigma}R$ is not semicommutative. On other hand, we have $(1,0)(0,1)=0$, but $(1,0)\sigma((0,1))=(1,0)(1,0)=(1,0)\neq 0$, so $R$ does not satisfying the condition $(\mathcal{C}_{\sigma}^1)$. Thus the condition $(\mathcal{C}_{\sigma}^1)$ in Corollary \ref{cor1} is not superfluous.
\end{ex}

\section*{acknowledgments}The author wishes to thank Professor Chan Yong Hong from Kyung Hee University, Korea for valuable remarks and helpful comments.

\end{document}